\documentclass[12pt]{amsart}

\usepackage{amsmath,amssymb,amscd,verbatim,color}
\newcommand{\mc}{\mathcal}

\newcommand{\sub}{\subseteq}

\newcommand{\subneq}{\subsetneq}

\newcommand{\ol}{\overline}

\newcommand{\lra}{\Leftrightarrow}

\newcommand{\ra}{\Rightarrow}

\newcommand{\sm}{\setminus}

\newcommand{\al}{\alpha}

\newcommand{\be}{\beta}

\newcommand{\wmax}{w\op{-Max}}
\newcommand{\tmax}{t\op{-Max}}
\newcommand{\Max}{\op{Max}}

\newcommand{\spec}{\op{Spec}}

\newcommand{\astspec}{\ast\op{-Spec}}
\newcommand{\astmax}{\ast\op{-Max}}

\newcommand{\op}{\operatorname}

\newcommand{\F}{\mc F}
\newcommand{\cP}{\mc P}
\newcommand{\cS}{\mc S}

\newtheorem{theorem}{Theorem}[section]

\newtheorem{lemma}[theorem]{Lemma}

\newtheorem{prop}[theorem]{Proposition}

\newtheorem{cor}[theorem]{Corollary}

\newtheorem{remark}[theorem]{Remark}

\theoremstyle{definition}

\begin{document}

	\title{Stability and Clifford regularity \\with respect to star operations}



\author{Stefania Gabelli and Giampaolo Picozza
}

\address{Dipartimento di Matematica, Universit\`{a} degli Studi Roma
Tre,
Largo S.  L.  Murialdo,
1, 00146 Roma, Italy}

\email{gabelli@mat.uniroma3.it, picozza@mat.uniroma3.it
}




\begin{abstract}  In the last few years, the concepts of stability and Clifford regularity have been fruitfully extended by using star operations. In this paper we deepen the  study of star stable and star regular domains and relate these two classes of domains  to each other.
\end{abstract}

\maketitle

\section*{Introduction}

Throughout all the paper, $R$ will be an integral domain and $K$ its field of fractions.
If $I$ is a nonzero fractional ideal of $R$, we call $I$ simply an \emph{ideal} and if $I\sub R$ we say that $I$ is an \emph{integral ideal}.

The \emph{ideal class semigroup} of $R$, here denoted by $\cS(R)$, consists of the isomorphism classes of the ideals of $R$. Clearly $R$ is a Dedekind domain if and only if $\cS(R)$ is a group.
By a well-known theorem of Clifford, a commutative semigroup $S$ is a disjoint union of groups if and only if each element $x\in S$ is \emph{von Neuman regular}, that is there exists an element $a\in S$ such that $x=x^2a$. Idempotent elements are regular. $S$ is called a \emph{Clifford semigroup} (respectively, a \emph{Boole semigroup}) if its elements are all regular (respectively, idempotent) and $R$ is called a \emph{Clifford regular domain} (respectively, a \emph{Boole regular domain}) if $\cS(R)$ is a Clifford (respectively, Boole)  semigroup.
Dedekind domains are trivial examples of Clifford regular domains. Zanardo and Zannier proved that all orders in quadratic fields are Clifford regular domains \cite{ZZ}  while  Bazzoni and Salce showed that all valuation domains are Clifford regular \cite{BS}.

A particular class of Clifford regular domains is given by stable domains. A domain is (\emph{finitely}) \emph{stable} if each (finitely generated) ideal  is invertible in its endomorphism ring. Stable
domains have been thoroughly investigated by B. Olberding  \cite{O3, O5, O1, O2}.

Since a valuation domain is stable if and only if it is strongly discrete \cite[Proposition 4.1]{O3}, not all Clifford regular domains are stable. On the other hand  Clifford regular domains are finitely stable, so that in the Noetherian case Clifford regularity coincides with stability  \cite[Theorem 3.1]{B3}.

The study of Clifford regular domains was carried on by S. Bazzoni \cite{B1, B2, B3, B4}. In particular in \cite{B3} she
  characterized integrally closed Clifford regular domains as Pr\"ufer domains with finite character. To this end, she established an interesting  relation between Clifford regularity and the local invertibility property. (A domain has the \emph{local invertibility property} if each locally invertible ideal is invertible.)  Bazzoni conjectured that a Pr\"ufer domain with the local invertibility property be of finite character. This conjecture was then solved in positive in \cite{HM} and the
local invertibility property and other related properties were later investigated by several authors \cite{McG, Za, ZD, FPT}.

Stability with respect to semistar operations was introduced and studied by the authors of this paper in \cite{GP}.

The first attempt to extend the notion of Clifford regularity in the setting of star operations is due to Kabbaj and Mimouni, who considered the class semigroup  of  $t$-ideals and extended the results known for Pr\"ufer domains to P$v$MDs. They also studied Clifford and Boole $t$-regularity for Noetherian and Mori domains \cite{KM1, KM2, KM3, KM4}.

Finally Halter-Koch, in the language of ideal systems, introduced Clifford $\ast$-regularity for star operations of finite type \cite{HK}. Among other results, he characterized Clifford $\ast$-regular P$\ast$MDs and in this setting he proved the star analog of Bazzoni's conjecture.

In this paper we deepen the  study of stability and regularity with respect to semistar operations and relate these two concepts  to each other.
For technical reasons, the most interesting consequences are obtained for (semi)star operations spectral and of finite type. In this case, we show that $\ast$-regularity implies that $\ast$ is the $w$-operation on $R$ (Corollary \ref{ast=w}).

In Section 1, we fix the notation and review some basic properties of $\ast$-stable and $\ast$-regular domains.

In Section 2, we study the transfer of $\ast$-stability and $\ast$-regularity to overrings. In particular we prove that several classes of  star compatible overrings inherit these properties. We also give conditions under which $\ast$-stability and $\ast$-regularity are $\ast$-local properties.

In Section 3, we consider local  domains. The problem of establishing when a local one-dimensional Clifford regular domain is stable was investigated by  P. Zanardo in \cite{Z} by means of valuation overrings. In terms of star operations, we prove that a local one-dimensional
Clifford regular domain $(R, M)$ is stable if and only if  the ($t$-)maximal ideals of the endomorphism ring $E:=(M:M)$ are all divisorial (Theorem \ref{onedimstable}).

Section 4 is devoted to the integrally closed case. First of all we prove that, for a star operation $\ast=\tilde{\ast}$ spectral and of finite type, the $\ast$-integral closure of a Clifford $\ast$-regular domain is a P$\ast$MD and that when $R$ is Boole $\ast$-regular it is a GCD-domain (Theorem \ref{wic}).  This allows us to characterize integrally closed $\ast$-regular and $\ast$-stable  domains by using results from \cite{GP} and \cite{HK} (Theorems \ref{pvmd} and \ref{stablepvmd}). As a consequence we get that an integrally closed $\ast$-regular domain is $\ast$-stable if and only if it is strongly $\ast$-discrete.
We finish this section by showing that, for $\ast=\tilde{\ast}$, in $\ast$-dimension one $\ast$-regularity and $\ast$-stability are equivalent if and only if the $\ast$-integral closure is a Krull domain (Theorem \ref{tdimone}).

Finally, in Section 5, we deal with the $\ast$-finite character.  The main result is Theorem \ref{FC}, which states that, for $\ast=\tilde{\ast}$, $\ast$-regular domains have  $\ast$-finite character. For the identity this was proved by S. Bazzoni in \cite[Theorem 4.7]{B4}, however our argument is different and more direct.  For  $\ast=\tilde{\ast}$, a domain is $\ast$-stable if and only if it is $\ast$-locally stable and it has the $\ast$-finite character  \cite[Theorem 1.9]{GP}, but it is not known if a similar result holds for $\ast$-regularity.  We show that,  for $\ast=\tilde{\ast}$ or $\ast=t$, this is true for a class of domains including domains of $\ast$-dimension one (Theorem \ref{teoIFC2} and Corollary \ref{teowM2}).
The proof is based on the observation that,  for this class of domains, the $\ast$-finite character is equivalent to the $\ast$-local $\ast$-invertibilty property (i.e., if $I^\ast R_M$ is principal for each $M\in \astmax(R)$, then $I$ is $\ast$-invertible).

\section{Preliminaries and notation}

Star operations, as the $v$-closure (or divisorial closure), the $t$-closure and the $w$-closure, are
an essential tool in modern multiplicative ideal theory for characterizing and investigating several classes of integral domains.
The consideration that some important operations on ideals, like the integral closure, satisfy almost all the properties of star operations led A. Okabe and R. Matsuda to introduce in 1994 the more general and flexible notion of semistar operation. The class of semistar operations includes the classical star operations and often provides a more appropriate context for approaching several questions of multiplicative ideal theory. For standard material about semistar
operations, see for example \cite{EFP}; here we recall some basic notions that we will use in the paper.

By $\ol{\mc F}(R)$ we denote the set of nonzero $R$-submodules of $K$ and by $\mc F (R)$  the
set of all ideals of $R$.  A \emph{semistar operation} (respectively, a \emph{star operation}) $\ast$ on $R$ is a map
$\ol{\mc F}(R)\rightarrow \ol{\mc F}(R)$ (respectively, $\mc F(R)\rightarrow \mc F(R)$), $I\mapsto I^{\ast}$, such that the following
conditions hold for each $0\not=a\in K$ and for each $I$, $J\in \ol{\mc F}(R)$ (respectively, $\mc F(R)$):

\begin{itemize}
\item[(i)]  $(aI)^\ast = aI^\ast$ (respectively, $(aI)^\ast = aI^\ast$ and $R=R^\ast$);

\item[(ii)] $I \subseteq I^\ast$, and $I \subseteq J \Rightarrow I^\ast
\subseteq J^\ast$;

\item[(iii)] $I^{\ast\ast} = I^\ast$.
\end{itemize}

A semistar operation $\ast$ is
called a \emph{semistar operation of finite type} if, for each $I
\in\ol{\mc F}(R)$, we have
$I^\ast = \bigcup \{J^\ast \, \vert \, J
\in \ol{\mc F}(R) \mbox{ finitely generated and } J \subseteq I\}$.

If $\ast$ is any semistar operation, the semistar operation $\ast_f$ defined by
$I^{\ast_{f}}:=\bigcup \{J^\ast \, \vert \, J \in \ol{\mc F}(R)
\mbox{ finitely generated and } J \subseteq I\}$,
for each $I \in \ol{\mc F}(R)$, is
the semistar operation of finite type \emph{associated to $\ast$}.
Clearly $\ast$ is of finite type if and only if $\ast=\ast_f$.

If $\ast$ is a semistar operation on $R$ such that $R^\ast=R$, $\ast$ is called a \emph{(semi)star operation}
on $R$ and its restriction to the set of ideals $\mc F(R)$
is a star operation on $R$, still denoted by $\ast$. Conversely, any star operation $\ast$ on $R$ can be extended to a (semi)star operation by setting $I^\ast=K$ for all $I\in \ol{\mc F}(R)\sm \mc F(R)$.

If $\ast$ is a semistar operation on $R$ and $D$ is an overring of $R$,
 the restriction of $\ast$ to the set of $D$-submodules of $K$ is a semistar operation on $D$,  here denoted by $\ast_{\vert_D}$ or by $\dot{\ast}$ when no confusion arises. When $D^ \ast =D$,  $\dot{\ast}$ is a (semi)star operation on $D$
\cite[Proposition 2.8]{fl01}.
Note that $\dot{\ast}$ shares many properties with $\ast$ (see
for instance \cite[Proposition 3.1]{giampa}); for example,
if $\ast$ is of finite type then $\dot{\ast}$ is of
finite type \cite[Proposition 2.8]{fl01}.

We will be mainly concerned with star operations and (semi)star operations.

If $\ast$ is a (semi)star operation, an ideal $I$ is a \emph{$\ast$-ideal} if $I=I^{\ast}$ and $I$ is called
\emph{$\ast$-finite} (or \emph{of finite type}) if $I^\ast=J^{\ast}=J^{\ast_f}$ for some finitely
generated ideal $J\in \ol{\mc F} (R)$.

A \emph{$\ast$-prime} is a prime ideal which is also a $\ast$-ideal and a \emph{$\ast$-maximal ideal} is a
$\ast$-ideal maximal in the set of proper integral $\ast$-ideals of $R$. We denote by
$\astspec(R)$ (respectively, $\astmax(R)$) the set of $\ast$-prime
(respectively, $\ast$-maximal) ideals of $R$. If $\ast$ is a (semi)star
operation of finite type, by Zorn's lemma each $\ast$-ideal is contained in
a $\ast$-maximal ideal, which is prime. In this case, $R=\bigcap _{M\in \astmax(R)}R_{M}$. We say that $R$
has \emph{$\ast$-finite character} if each nonzero element of $R$ is contained in
at most finitely many $\ast$-maximal ideals.

When $\ast $ is of finite type, a minimal prime of a $\ast$-ideal is
a $\ast$-prime. In particular, any minimal prime over a nonzero
principal ideal (in particular any height-one prime) is a $\ast$-prime, for any star operation $\ast$ of finite
type.
We say that $R$ has \emph{$\ast$-dimension
one} if each $\ast$-prime ideal has height one.

The \emph{identity} is a (semi)star operation denoted by $d$, $I^{d}:=I$
for each $I\in \ol{\mc F} (R)$. Two nontrivial (semi)star operations which have been
intensively studied in the literature are the \emph{$v$-operation} (or \emph{divisorial closure}) and the \emph{$t$-operation}.
The divisorial closure of $I\in \ol{\mc F} (R)$ is defined by setting $
I^{v}:=(R:(R:I))$, where for any $I, J\in \ol{\mc F} (R)$ we set $(J\colon I):=\{x\in K\,:\, xI
\subseteq J\}$. A $v$-ideal of $R$ is also called a \emph{divisorial ideal}.
The $t$
-operation is the (semi)star operation of finite type associated to $v$ and is therefore defined by setting
$I^t:= \bigcup \{J^v \, \vert \, J
\in \ol{\mc F}(R) \mbox{ finitely generated and } J \subseteq I\}$.

If $\mc F \subseteq \spec (R)$ is a \emph{defining family} for $R$, that is a family of pairwise incomparable prime ideals such that $R=\bigcap_{P \in \mc F} R_P$, and $\ast_P$ is a (semi)star operation on $R_P$, for each $P \in \mc F$, then $\ast := \wedge_{_{P \in\mc F}} \ast_P$, defined by $I \mapsto I^\ast := \bigcap_{P \in \mc F} (IR_P)^{\ast_P}$, for all $I\in\ol{\mc F}(R)$, is a (semi)star operation on $R$ \cite[Theorem 2]{DDA}. If $d_P$ is the identity on $R_P$,  we denote by $\ast_\mc F$ the (semi)star operation $\wedge_{_{P \in\mc F}} d_P$, defined by $I \mapsto I^{\ast_\mc F} := \bigcap_{P \in \mc F} IR_P$.
A (semi)star operation $\ast$ is called \emph{spectral} if $\ast=\ast_\mc F$ for some defining family of prime ideals $\mc F$ of $R$. It is easy to see that, if $\ast$ is spectral, a prime ideal is a $\ast$-ideal if and only if $P^\ast \subsetneq R$.

Clearly when $\mc F=\Max(R)$ we have $\ast_\mc F=d$.
If $\ast$ is a (semi)star operation of finite type on $R$, taking  $\mc F = \astmax(R)$, the induced (semi)star operation $\ast_\mc F$ is here denoted by $\tilde{\ast}$.  The (semi)star operation $\tilde{t}$ is usually denoted by $w$.

It is known that $\ast
= \tilde{\ast}$ if and only if $\ast$ is spectral and of finite
type, if and only if $\ast$ is of finite type and $(I\cap
J)^\ast = I^ \ast \cap J^\ast$, for $I$, $J \in \ol{\mc F}(R)$
\cite[Corollary 3.9 and Proposition 4.23]{FH2000}.

If $\ast_1$ and $\ast_2$ are (semi)star operations on $R$, we say
that $\ast_1 \leq \ast_2$ if $I^{\ast_1} \subseteq
I^{\ast_2}$, for each $I \in \ol{\mc F} (R)$. This is equivalent to the
condition that $(I^{\ast_1})^{\ast_2}= (I^{\ast_2})^{\ast_1} =
I^{\ast_2}$.
If $\ast_1 \leq \ast_2$, then $(\ast_1)_f \leq (\ast_2)_f$ and $\widetilde{\ast_1} \leq \widetilde{\ast_2}$.
Also, for each (semi)star
operation $\ast$, we have $d\leq \ast\leq v$ (so that $\ast_f\leq t$ and $\tilde{\ast}\leq w$)
 and $\tilde{\ast} \leq \ast_f \leq \ast$ (so that $w\leq t\leq v$).

 For any star operation $\ast$, the set of $\ast$-ideals of $R$, denoted by $\mc F_\ast(R)$, is a
semigroup under the $\ast$-multiplication, defined by $(I, J)\mapsto (IJ)^\ast$, with
unit $R$. An ideal $I\in \mc F (R)$ is called $\ast$-invertible if $I^\ast$ is
invertible in $\mc F_\ast(R)$, equivalently
$(I(R:I))^\ast=R$.
If $\ast$ is a star operation of finite type, then $I$ is $\ast$-invertible if and only if $I$ is $\ast$-finite and $I^\ast R_M$ is principal for each $M\in \astmax(R)$ \cite[Proposition 2.6]{K}.

\smallskip
Given an ideal $I$ of $R$, we denote by $E(I):=(I:I)$ the endomorphism ring of $I$ and by $T(I):= I(E(I):I)=I(I:I^2)$ the trace of $I$ in $E(I)$.

If $\ast$ is a star operation on $R$, it is easy to
see that $E(I^\ast)^\ast=E(I^\ast)$. Thus the
restriction of $\ast$ to the set of fractional ideals of
$E(I^\ast):=(I^\ast:I^\ast)$ is a star operation on $E(I^\ast)$,
denoted by $\dot{\ast}:=\ast_{\vert_E}$.

 We say that
an ideal $I$ of $R$ is \emph{$\ast$-stable} if $I^\ast$ is
$\dot{\ast}$-invertible in $E(I^\ast)$ and that $R$ is
\emph{$\ast$-stable} (respectively, \emph{finitely $\ast$-stable}) if each ideal (respectively, each finitely generated ideal) of $R$ is $\ast$-stable. We also say that $I$ is \emph{strongly $\ast$-stable} if $I^\ast$ is principal in $E(I^\ast)$ and that $R$ is \emph{strongly $\ast$-stable} if each ideal is strongly $\ast$-stable.

 Denoting by $\cP(R)$ the group of principal ideals of $R$, the quotient semigroup $\cS_\ast(R):=\F_\ast(R)/\cP(R)$ is called the \emph{$\ast$-Class semigroup} of $R$. We say that $R$ is \emph{Clifford $\ast$-regular}, or simply \emph{$\ast$-regular}, if $\cS_\ast(R)$ is a Clifford semigroup. This means that  for each ideal $I$, the class $[I^\ast]\in \cS_\ast(R)$ is (von Neumann) regular. Note that this is equivalent to say that $I^\ast$ is (von Neumann) regular in $\F_\ast(R)$, that is $I^\ast=(I^2J)^\ast$, for some nonzero ideal $J$ of $R$.
 If  $[I^\ast]$ is regular in $\cS_\ast(R)$, we say that $I$ is \emph{$\ast$-regular}. If $[I^\ast]$ is idempotent, that is $[I^\ast]=[I^\ast]^2=[(I^2)^\ast]$ (equivalently $(I^2)^\ast=xI^\ast$  for a nonzero $x\in K$) we say that $I$ is \emph{Boole $\ast$-regular}  and if each $[I^\ast]$ is idempotent, that is $\cS_\ast(R)$ is a \emph{Boole semigroup}, we say that $R$ is \emph{Boole $\ast$-regular}.
Clearly Boole $\ast$-regularity implies Clifford $\ast$-regularity.

\begin{remark} \rm One could define Clifford star regularity more generally for a semistar operation. In this case, for a semistar operation $\ast$, set $\mathcal{F}_\ast(R) := \{I^\ast ; I \in \mathcal F(R)\}$ and $\cP_\ast(R) := \{xR^\ast ; x \in K\sm \{0\}\}=\cP(R^\ast)$ and consider the quotient semigroup $\mathcal{F}_\ast(R) / \cP_\ast(R)$. It turns out that this is exactly the class semigroup of the overring $R^\ast$ with respect to the (semi)star operation $\dot{\ast}$. Thus, $R$ is Clifford $\ast$-regular if and only if $R^\ast$ is Clifford $\dot{\ast}$-regular. Since the transfer of properties of semistar operations between a domain and its overrings is well-understood, we study Clifford regularity only in the case of star operations.
\end{remark}

In the following two lemmas, we restate some basic results on regularity proved by F. Halter-Koch for star operations of finite type in the language of ideal systems on monoids \cite[Proposition 4.2]{HK}. We note that the assumption that $\ast$ be of finite type is not necessary.

\begin{lemma}\label{lemma1}
Let $R$ be a domain and $\ast$ a star operation on $R$. For an ideal $I$ of $R$,
the following conditions are equivalent:

\begin{enumerate}
\item[(i)] $I$ is Clifford $\ast$-regular;

\item[(ii)] $I^\ast= (I^2(I^\ast:I^2))^\ast$;

\item[(iii)] $I^\ast = (IT(I^\ast))^\ast$.
\end{enumerate}

\end{lemma}

\begin{lemma} \label{lemmaX}
Let $R$ be a domain, $\ast$ a star operation on $R$ and assume that $I$ is Clifford $\ast$-regular. Then:

\begin{enumerate}

\item[(1)] If $X$ is an ideal of $R$ such that $I^\ast = (I^2X)^\ast$, then $(IX)^\ast = T(I^\ast)^\ast$.

\item[(2)] $(T(I^\ast)^2)^\ast = (T(I^\ast)^2)^{\dot{\ast}} = T(I^\ast)^\ast$.

\item[(3)] $E(I^\ast) = (T(I^\ast)^\ast:T(I^\ast)) = (E(I^\ast):T(I^\ast))$.

\end{enumerate}
\end{lemma}

\begin{remark} \label{rem1} \rm (1) By Lemma \ref{lemma1}, we see that if $\ast_1\leq \ast_2$, then $\ast_1$-regularity implies $\ast_2$-regularity.

(2)
If $\ast=\tilde{\ast}$ is spectral and of finite type and $I$ is a finitely generated ideal of $R$, then $(J^\ast:I^\ast)=(J:I)^\ast$, for all ideals $J$. Hence a finitely generated ideal $I$ is $\ast$-stable if and only if  $E(I)^\ast=(T(I))^\ast$ and $I$ is $\ast$-regular if and only if $I^\ast=(I^2(E(I)^\ast:I))^\ast=(IT(I))^\ast$.
\end{remark}

\begin{prop} \label{prop1} Let $I$ be an
ideal of $R$ and, for a star operation $\ast$ on $R$, set $E:=E(I^\ast)$.
\begin{enumerate}
\item[(1)] If $I$  is $\ast$-stable, then  $I$  is $\ast$-regular. Hence a $\ast$-stable domain is Clifford $\ast$-regular.

\item[(2)] If $I$ is $\ast$-regular, then  $I^\ast$ is $v_E$-invertible in $E$ and if, in addition, $I$ is finitely generated, then $I^\ast$ is  $t_E$-invertible in $E$ (where $v_E$ and $t_E$ denote respectively  the $v$-operation and the $t$-operation on $E$).

\item[(3)] $I$ is strongly $\ast$-stable if and only if $I$ is Boole $\ast$-regular and $\ast$-stable.
Hence a strongly $\ast$-stable domain is precisely a Boole $\ast$-regular $\ast$-stable domain.
\end{enumerate}

\end{prop}
\begin{proof} (1) is an easy calculation \cite[Proposition 4.6]{HK}.

(2) Set $T:= T(I^\ast)$.
Since $E=(E:T)$ (Lemma \ref{lemmaX}(3)), it follows that $T^{v_E}=E$.

If $I$ is finitely generated, then $T$ is $\ast$-finite in $R$ \cite[Proposition 4.2]{HK}. Write $T^\ast=F^\ast$, with $F$ finitely generated. Then $T^{\dot{\ast}}=(FE)^{\dot{\ast}}$ is $\dot{\ast}$-finite in $E$ and, since  $\dot{\ast}\leq t_E$, $T^{t_E}=(FE)^{t_E}=(FE)^{v_E}=T
^{v_E}=E$.

(3) If $I$ is strongly $\ast$-stable, say $I^\ast=xE$ with $x \in K\sm\{0\}$, then $(I^2)^\ast=xI^\ast$. Thus $I$ is Boole $\ast$-regular and clearly $I^\ast$ is ($\dot{\ast}$-)invertible in $E$. Conversely, if $I$ is Boole $\ast$-regular and $I^\ast$ is $\dot{\ast}$-invertible in $E$ we have $(I^2)^\ast=xI^\ast$, for some $x\in K\sm\{0\}$, and  $I^\ast=(I^2(I^\ast:I^2))^\ast=x(I^\ast(E:I^\ast))^{\dot{\ast}}=xE$. Hence $I$ is strongly $\ast$-stable.
\end{proof}

For the identity the next result  is proved in \cite[Lemma 2.1]{B3}.

\begin{prop} \label{fingen} Let  $\ast:=\ast_\mc F$ be a spectral star operation on $R$.
If $I$ is a $\ast$-finite $\ast$-regular ideal of  $R$, then  $I$ is $\ast$-stable.
\end{prop}
\begin{proof}
 Let $I^\ast=J^\ast$, with $J$ finitely generated. Setting $E:= E(I^\ast)=E(J^\ast)$ and $T:= T(I^\ast)=T(J^\ast)$,
 we have $I^\ast=(JE)^\ast=(JT)^\ast$.
 Localizing at a prime $P\in \mc F$, we obtain $J_PE_P=J_PT_P$.
 Since $J$ is finitely generated and $T \sub E$, it must be $T_P=E_P$ for each $P\in \mc F$ \cite[Corollary 6.4(b)]{g1}.
 Hence $T^\ast= (I^\ast(E:I^\ast))^\ast=E^\ast=E$. We conclude that $I$ is $\ast$-stable.
\end{proof}

\begin{cor} \label{ast=w} Let  $\ast:=\tilde{\ast}$ be a star operation on $R$ spectral and of finite type. If $R$ is Clifford $\ast$-regular, then $\ast=w$.
\end{cor}
\begin{proof}  If $\ast\neq w$ there exists a  $\ast$-maximal ideal $M$ of $R$ such that $M^t=R$. Let $J\sub M$ be a finitely generated ideal such that $J^t=R$. Then $E(J^\ast) \sub E(J^t)=R$. Whence $E(J^\ast)=R$ and, since $J$ is finitely $\ast$-stable (Proposition \ref{fingen}),  $J$ is $\ast$-invertible in $R$. It follows that $J^\ast=J^v=J^t=R$, which is a contradiction since $J^\ast\sub M$.
\end{proof}

\begin{remark} \rm  (1) It is not always true that if  $\ast$ is of finite type and $R$ is Clifford $\ast$-regular, then $\ast=t$. Indeed this is not even true for $\ast$-stable domains and $\ast = \tilde{\ast}$ \cite[Remark 1.7(2)]{GP}.

(2) A Clifford $t$-regular domain need not be finitely $t$-stable.  In fact there exist Noetherian Clifford $t$-regular (indeed Boole $t$-regular) domains with  $t$-ideals of height greater than one \cite[Example 2.4]{KM4}, while a $t$-stable Noetherian domain has $t$-dimension one \cite[Corollary 3.3]{GP2}.

(3)  A Clifford $t$-regular domain is not always Clifford $w$-regular. In fact any pseudo-valuation domain is a Clifford $t$-regular domain \cite[Theorem 2.7]{KM2},  but, since local integrally closed Clifford regular local domains are valuation domains \cite[Theorem 3]{BS},  an integrally closed PVD, which is not  a valuation domain, is never Clifford regular and so it is never Clifford $w$-regular, since in a PVD $w = d$.
\end{remark}

\section{Overrings}

Let $R\sub D$ be an extension of domains. If $\ast_1$ is a star operation on $R$ and $\ast_2$ is a star operation on $D$, the extension is \emph{$(\ast_1, \ast_2)$-compatible} (or  \emph{$\ast_1$ and $\ast_2$ are compatible}) if $(ID)^{\ast_2}=(I^{\ast_1} D)^{\ast_2}$ for every ideal $I$ of $R$ \cite[Section 4]{A}. It is easily seen that, if $\ast_1$ and $\ast_2$ are of finite type, it suffices that this compatibility condition is satisfied by finitely generated ideals.

If $\ast=d, w, t, v$, we will denote by $\ast_D$ the corresponding star-operation on $D$ and
will say that the extension is \emph{$\ast$-compatible} if it is $(\ast, \ast_D)$-compatible.

Also recall that $D$ is \emph{$t$-linked over $R$} if $(Q\cap R)^t \subneq D$ for each prime $t_D$-ideal of $D$ with $Q\cap R\neq (0)$.
It is known that an extension is $t$-linked if and only if it is $w$-compatible  \cite[Proposition 3.10]{Jesse}.

One can generalize in the natural way this definition of compatibility to semistar operations, requiring that the compatibility relation holds for all $R$-submodules of the quotient field $K$ of $R$. Again, for semistar operations of finite type, it suffices to verify the  condition on finitely generated ideals. In particular the notion of $t$-compatibility (respectively,  $w$-compatibity) is the same if one considers $t$ (respectively, $w$) as a star operation or as a (semi)star operation \cite{EF}.

\begin{prop} \label{lemma:comp}
Let $R \sub D$ be an extension of domains with the same quotient field, $\ast_1$  a semistar operation on $R$ and $\ast_2$  a semistar operation on $D$. The following conditions are equivalent:

\begin{enumerate}
\item[(i)]  The extension is $(\ast_1, \ast_2)$-compatible.

\item[(ii)] $\dot{\ast_1} \leq \ast_2$.
\end{enumerate}

If moreover $\ast_1$ and $\ast_2$ are of finite type, the two conditions above are equivalent to:

\begin{enumerate}
\item[(iii)] $F^{\dot{\ast_1}} \sub F^{\ast_2}$ for all (finitely generated) ideals $F$ of $D$.
\end{enumerate}
\end{prop}

\begin{proof}
(i) $\Rightarrow$ (ii) Let $F$ be a $D$-submodule of $K$. Obviously $F$ is also an $R$-module, so $F^{\dot{\ast_1}} = F^{\ast_1} \sub (F^{\ast_1}D)^{\ast_2} = (FD)^{\ast_2} = F^{\ast_2}$. Thus $\dot{\ast_1} \leq \ast_2$.

(ii) $\Rightarrow$ (i) Let $F$ be an $R$-submodule of $K$. Then $(FD)^{\ast_2} \subseteq (F^{\ast_1}D)^{\ast_2} \subseteq ((F^{\ast_1}D)^{\dot{\ast_1}})^{\ast_2} = ((FD)^{\dot{\ast_1}})^{\ast_2} = (FD)^{\ast_2}$. Thus $(FD)^{\ast_2} = (F^{\ast_1}D)^{\ast_2}$ and so the extension is $(\ast_1, \ast_2)$-compatible.

(ii) $\Leftrightarrow$ (iii) is obvious, since finite type semistar operations are completely determined by their behavior on finitely generated ideals.
\end{proof}

\begin{cor} \label{cor:t-comp}
Let $R \sub D$ be an extension of domains with the same quotient field. Then:

\begin{enumerate}

\item[(1)] The extension is $t$-compatible if and only if $D^t = D$ (where $t$ is the $t$-semistar operation of $R$).

\item[(2)] The extension is $w$-compatible (equivalently, $D$ is $t$-linked over $R$) if and only if $D^w = D$ (where $w$ is the $w$-semistar operation of $R$).

\end{enumerate}
\end{cor}
\begin{proof}
(1) If the extension is $t$-compatible, from Proposition \ref{lemma:comp} we have $D^{t} = D^{\dot{t}} \subseteq D^{t_D} = D$. Conversely, if $D^{\dot{t}}  = D^{t}= D$, the restriction of $\dot{t}$ to the fractional ideals of $D$ is a star operation of finite type on $D$. So, it is smaller than the $t$-operation of $D$, which is the largest star operation of finite type on $D$. Hence, by Proposition \ref{lemma:comp}((iii) $\Rightarrow$ (i)), the extension is $t$-compatible.

(2) This result can be proved easily with a similar argument, since the $w$-operation is the largest star operation spectral and of finite type. However, since $w$-compatible is equivalent to $t$-linked \cite[Proposition 3.10]{Jesse}, (2) is well-known and follows by \cite[Proposition 2.13(a)]{DHLZ}.
\end{proof}

Several well-known facts about compatibility follow easily from Corollary \ref{cor:t-comp}. Since $w \leq t$,  a $t$-compatible extension is also $w$-compatible (cf. \cite[p. 1475]{Jesse}). Moreover, since for a flat $R$-module $F$ contained in the quotient field of $R$, $F^t = F$ (the proof of Theorem 1.4 in \cite{PT} works also in this more general case and not only for ideals), a flat extension is $t$-compatible (cf. \cite[Lemma 3.4(iii)]{K}).  In addition, it is clear from  Corollary \ref{cor:t-comp} that the endomorphism rings of $t$ and $v$-ideals are $t$-compatible and the endomorphism rings of $w$-ideals are $w$-compatible (i.e., $t$-linked).

If  $\mc S$ is a multiplicative system of ideals of $R$, we denote by $R_\mc S:=\bigcup \{(R:I)\,;\; I\in \mc S\}$  the \emph{generalized ring of fractions of $R$ with respect to $\mc S$}. Generalized rings of fractions are $t$-compatible \cite[Lemma 3.4(iii)]{K}.

An extension $R \subseteq D$ is \emph{locally $t$-compatible} if for every prime ideal $P$ of $R$, the extension $R_P \subseteq D_{R\setminus P}$ is $t$-compatible \cite[p. 1477]{Jesse}.  The following result solves a problem posed in \cite[p. 1481]{Jesse}, at least for an extension of domains with the same quotient field.

\begin{prop}
Let $R \subseteq D$ be a locally $t$-compatible extension of domains with the same quotient field. Then the extension is $t$-compatible.
\end{prop}
\begin{proof}
For every prime $P$ of $R$ we have $D^{t_R} \subseteq (D_{R \setminus P}^{t_R})^{t_{R_P}} = (D_{R \setminus P})^{t_{R_P}}$ by Proposition \ref{lemma:comp} ((i) $\Leftrightarrow$ (ii)), since the extension $R\subseteq R_P$ is $t$-compatible. By the hypothesis, the extension $R_P \subseteq D_{R\setminus P}$ is $t$-compatible, so $(D_{R \setminus P})^{t_{R_P}} = D_{R \setminus P}$ by Corollary \ref{cor:t-comp}. Then $D^{t_R} \subseteq \bigcap_{P \in \spec(R)} DR_P = D$ and the extension $R \subseteq D$ is $t$-compatible, again by Corollary \ref{cor:t-comp}.
\end{proof}

\begin{lemma} \label{lemma2} Let $R\sub D$ be a $(\ast_1, \ast_2)$-compatible extension of domains, where $\ast_1$ is a star operation on $R$ and $\ast_2$ is a star operation on $D$. If $I$ is a $\ast_1$-regular (respectively, Boole $\ast_1$-regular, $\ast_1$-stable, strongly $\ast_1$-stable)  ideal of $R$, then $ID$ is a $\ast_2$-regular (respectively, Boole $\ast_1$-regular, $\ast_1$-stable, strongly $\ast_1$-stable)   ideal of $D$. Hence, if $R$ is Clifford $\ast_1$-regular (respectively, Boole $\ast_1$-regular, $\ast_1$-stable, strongly $\ast_1$-stable)  and each $\ast_2$-ideal of $D$ is of type $(ID)^{\ast_2}$, for some ideal $I$ of $R$, $D$ is Clifford $\ast_2$-regular (respectively, Boole $\ast_1$-regular, $\ast_1$-stable, strongly $\ast_1$-stable).
\end{lemma}
\begin{proof}  Let $I$ be a  $\ast_1$-regular ideal of $R$. Hence $(ID)^{\ast_2}=(I^{\ast_1} D)^{\ast_2}=((I^2(I^{\ast_1}:I^2))^{\ast_1} D)^{\ast_2}=(I^2(I^{\ast_1}:I^2)D)^{\ast_2}\sub((ID)^2((ID)^{\ast_2}:(ID)^2))^{\ast_2}\sub (ID)^{\ast_2}$. Whence $ID$ is $\ast_2$-regular.

If $I$ is Boole $\ast_1$-regular, $I^{\ast_1} = x(I^2)^{\ast_1}$, for some nonzero $x$ in $K$. So, by $(\ast_1, \ast_2)$-compatibility, we have $(ID)^{\ast_2} = (I^{\ast_1} D)^{\ast_2} = (x(I^2)^{\ast_1} D)^{\ast_2} = (xI^2 D)^{\ast_2} = x((ID)^2)^{\ast_2}$. Hence $ID$ is Boole $\ast_2$-regular.

If $I$ is $\ast_1$-stable,  $1 \in R=(I^{\ast_1}(I^{\ast_1}:(I^{\ast_1})^2))^{\ast_1} \subseteq ((I^{\ast_1}(I^{\ast_1}:(I^{\ast_1})^2))^{\ast_1}D)^{\ast_2} \subseteq ((ID)^{\ast_2}((ID)^{\ast_2}:(ID^{\ast_2})^2))^{\ast_2} $ by repeated use of  $(\ast_1, \ast_2)$-compatibility. So, $ID$ is $\ast_2$-stable.

If $I$ is strongly $\ast_1$-stable, then $I$ is Boole $\ast_1$-regular and so $ID$ is Boole $\ast_2$-regular. Since $ID$ is also $\ast_2$-stable by the above, then $ID$ is strongly $\ast_2$-stable (Proposition \ref{prop1}(3)).
\end{proof}

Among other cases, each ideal $I$ of the overring $D$ is extended from $R$ (i.e., $I = JD$ for a (fractional) ideal $J$ of $R$) if $D$ is a fractional overring, or $D$ is flat, or $D$ is Noetherian \cite{sega}.  Note that, since $D\sub K$, each finitely generated ideal of $D$ is a fractional ideal of $R$ and so it is trivially extended from $R$.

The next proposition was proved for regularity in \cite[Lemma 2.2]{KM1} and \cite[Proposition 4.1]{sega}.

\begin{prop} \label{overrings1} Let $R\sub D\sub K$, $\ast_1$ a star operation on $R$ and $\ast_2$ a star operation  on $D$.  Assume that the extension is $(\ast_1, \ast_2)$-compatible and that $R$ is Clifford $\ast_1$-regular (respectively, Boole $\ast_1$-regular, $\ast_1$-stable, strongly $\ast_1$-stable). Assume also that one of the following conditions holds:
\begin{itemize}
\item[(a)] $D$ is flat over $R$ (for example $D$ is a ring of fractions of $R$);
\item[(b)]  $D$ is a fractional overring of $R$.
\item[(c)] Each ideal of $D$ is $\ast_2$-finite.
\end{itemize}
Then $D$ is Clifford $\ast_2$-regular (respectively, Boole $\ast_2$-regular, $\ast_2$-stable, strongly $\ast_2$-stable).
\end{prop}
\begin{proof} In the cases (a) and (b), each ideal of $D$ is extended from $R$ \cite[Propositions 3.8 and 3.4]{sega}. By hypothesis the extension is $(\ast_1, \ast_2)$-compatible. Hence we can apply Lemma \ref{lemma2}.

(c) Let $I$ be an ideal of $D$. By hypothesis $I^{\ast_2} = J^{\ast_2}$ for some finitely generated ideal $J$ of $D$. Clearly $J$ is a fractional ideal of $R$, thus it is $\ast_1$-regular in $R$ and so $\ast_2$-regular in $D$ by Lemma \ref{lemma2}. Hence also $I$ is $\ast_2$-regular. The other cases are similar.
\end{proof}

\begin{cor} \label{overrings} Let $R\sub D\sub K$, $\ast = d, w, t$ on $R$ and $\ast_D$ the corresponding star operation on $D$. Assume that $R$ is Clifford $\ast$-regular (respectively, Boole $\ast$-regular, $\ast$-stable, strongly $\ast$-stable) and that one of the following conditions holds:
\begin{itemize}
\item[(a)] $D$ is flat over $R$ (for example $D$ is a ring of fractions of $R$);
\item[(b)]  $D$ is a $\ast$-compatible fractional overring;
\item[(c)]  $D$ is $\ast$-compatible and Noetherian.
\end{itemize}
Then $D$ is Clifford $\ast_D$-regular (respectively, Boole $\ast_D$-regular,  $\ast_D$-stable, strongly $\ast_D$-stable).
\end{cor}
\begin{proof} Apply Proposition \ref{overrings1} (for (a), recall that a flat extension is always $t$-compatible).
\end{proof}

For the $d$ and the $w$-operations the statements for stability in Corollary \ref{overrings} are particular cases of more general results (cf. \cite[Theorem 5.1]{O2} and \cite[Corollary 2.2]{GP}).

\begin{remark} \label{rm1}
\rm Note that, if $\ast$ is a star operation on $R$, an overring $D$ of $R$ is a $\ast$-ideal if and only if $D=E(I)$ for some integral $\ast$-ideal $I$ of $R$. Indeed, if $D$ is a fractional overring of $R$, $D = x^{-1} I$ for some integral ideal $I$ of $R$ and $x\in R$. If, in addition,  $D= D^\ast =  x^{-1}I^{\ast}$, we have $I = I^\ast$. Moreover, $(I:I) = ( x^{-1} I :  x^{-1} I) = (D:D) = D$. The converse is clear.

In particular, by Lemma \ref{lemma:comp}, $D$ is  $t$-compatible (respectively, $t$-linked) if and only if $D=(I:I)$ for some $t$-ideal (respectively, $w$-ideal) $I$ of $R$.
\end{remark}

If $\mc S$ is a multiplicative system of ideals of $R$, the \emph{saturation of $\mc S$}, here denoted by  $\ol{\mc S}$, is the multiplicative system of ideals consisting of all the ideals of $R$ containing some ideal in $\mc S$. Since $R_\mc S=R_{\ol{\mc S}}$, we can always assume that $\mc S$ is saturated, that is $\mc S= \ol{\mc S}$. We say that  $\mc S$ is \emph{$v$-finite} if each $t$-ideal $I\in \ol{\mc S}$ contains a finitely generated ideal $J$ such that $J^t=J^v \in \ol{\mc S}$ \cite{G}. Clearly finitely generated multiplicative systems of ideals are $v$-finite.

We denote by $t_\mc S$ the $t$-operation on the generalized ring of fractions $R_\mc S:=\cup\{(R:I); I\in \mc S\}$.

\begin{prop} \label{prop:generalized}  Let $\mc S$ be a $v$-finite multiplicative system of ideals. If $R$ is Clifford $t$-regular (respectively, Boole $t$-regular, $t$-stable, strongly $t$-stable), then $R_\mc S$ is Clifford $t_\mc S$-regular (respectively, Boole $t_\mc S$-regular, $t_\mc S$-stable, strongly $t_\mc S$-stable).
\end{prop}
\begin{proof} If $\mc S$ is $v$-finite, each $t$-ideal of $R_\mc S$ is of type $(IR_\mc S)^{t_\mc S}$ for some ideal $I$ of $R$ \cite[Proposition 1.8]{G}. Since  generalized rings of fractions are $t$-compatible, we conclude by Lemma \ref{lemma2}.
\end{proof}

\begin{remark} \rm If $D$ is a $t$-flat overring of $R$ (i.e., if $D_M = R_{M \cap R}$ for each $t$-maximal ideal $M$ of $D$) then $D$ is a generalized ring of fractions with respect to a $v$-finite multiplicative system of ideals of $R$ \cite[Theorem 2.6]{E}. Hence Proposition \ref{prop:generalized} implies a $t$-version of Corollary \ref{overrings}(a).\end{remark}

\begin{prop} \label{prop:loc}
Let $\mc F \subseteq \spec (R)$ be a defining family for $R$ and let $\ast := \wedge_{_{P \in \mc F}} \ast_P$, where $\ast_P$ is a star operation on $R_P$  for each $P\in \mc F$. The following conditions are equivalent for an ideal $I$ of $R$:

\begin{enumerate}
\item[(i)] $I$ is Clifford $\ast$-regular (respectively, $\ast$-stable);

\item[(ii)] $IR_P$ is Clifford $\ast_P$-regular (respectively, $\ast_P$-stable) and $(I (I^\ast:I^2)R_P)^{\ast_P} = (IR_P ((IR_P)^{\ast_P} : (I R_P)^2))^{\ast_P}$, for each $P \in \mc F$.
\end{enumerate}
\end{prop}
\begin{proof}
(i) $\Rightarrow$ (ii) Since the extension $R \sub R_P$ is $(\ast, \ast_P)$-compatible,  for each $P \in \mc F$, $IR_P$ is Clifford $\ast_P$-regular by Lemma \ref{lemma2}.
Then, since $I^\ast = (I^2(I^\ast : I^2))^\ast$, we have $(IR_P)^{\ast_P} = (I^\ast R_P)^{\ast_P} = ((I^2(I^\ast : I^2))^\ast R_P)^{\ast_P} = ( (IR_P)^2 (I^\ast : I^2)R_P)^{\ast_P}$. We then apply Lemma \ref{lemmaX}(1) to the Clifford $\ast_P$-regular ideal $I R_P$, with $X = (I^\ast : I^2)R_P$  and obtain $( IR_P (I^\ast : I^2)R_P)^{\ast_P} = (IR_P ( (IR_P)^{\ast_P} : (IR_P)^2))^{\ast_P}$.

Assume now that $I$ is $\ast$-stable and let $P \in \mc F$. By Proposition \ref{prop1}(1) $R$ is Clifford $\ast$-regular, so the second condition in (ii) is satisfied. Thus we need only show that $IR_P$ is $\ast_P$-stable. Since $I$ is $\ast$-stable, we have that $(I^\ast (I^\ast : I^2))^\ast = (I^\ast : I^\ast)$. It follows that $((I^\ast (I^\ast : I^2))^\ast R_P)^{\ast_P} = ((I^\ast : I^\ast)R_P)^{\ast_P}$. Thus $1 \in  ((I^\ast (I^\ast : I^2))^\ast R_P)^{\ast_P}$.  Hence, by the second condition of (ii), $ 1 \in (IR_P ( (IR_P)^{\ast_P} : (I R_P)^2))^{\ast_P}$ and so $(IR_P)^{\ast_P}$ is $\dot{\ast_P}$-invertible in $((IR_P)^{\ast_P} : (IR_P)^{\ast_P})$. Thus $IR_P$ is $\ast_P$-stable.

(ii) $\Rightarrow$ (i)  Since $R_P$ is Clifford $\ast_P$-regular, $(IR_P)^{\ast_P} = ((IR_P)^2  ((I R_P)^{\ast_P} : (IR_P)^2))^{\ast_P} = (I^2 (I^\ast : I^2) R_P)^{\ast_P}$, for each $P \in \mc F$. So, $I^\ast = \bigcap_P (IR_P)^{\ast_P} = \bigcap_P (I^2 (I^\ast:I^2)R_P)^{\ast_P} =  (I^2 (I^\ast:I^2))^{\ast}$, and $I$ is Clifford $\ast$-regular.

Assume now that $IR_P$ is $\ast_P$-stable for all $P \in \mc F$, i.e. that $(IR_P ( (IR_P)^{\ast_P} : (I R_P)^2))^{\ast_P} = ( (IR_P)^{\ast_P} : (IR_P)^{\ast_P} )$. By the hypothesis $(I (I^\ast:I^2)R_P)^{\ast_P} = ( (IR_P)^{\ast_P} : (IR_P)^{\ast_P})$. So $(I^ \ast (I^\ast : I^2 ))^\ast = \bigcap_{P \in \mathcal F} ( (IR_P)^{\ast_P} : (IR_P)^{\ast_P}) =  \bigcap_{P \in \mathcal F} ( (IR_P)^{\ast_P} : I) = ( \bigcap_{P \in \mathcal F} (IR_P)^{\ast_P} : I ) = (I^ \ast : I) = (I^\ast : I^\ast)$ and $I$ is $\ast$-stable.
\end{proof}

\begin{cor} \label{loc-t}
Let $R$ be an integral domain with the $t$-finite character. Then the following conditions are equivalent for an ideal $I$ of $R$:
\begin{enumerate}
\item[(i)] $I$ is Clifford $t$-regular (respectively, $t$-stable);

\item[(ii)] $IR_M$ is Clifford $t_M$-regular (respectively, $t_M$-stable) and $(I(I^t:I^2)R_M)^{t_M} = (IR_M((I R_M)^{t_M}:(IR_M)^2))^{t_M}$, for each ideal $I$ and $M\in \tmax(R)$.
\end{enumerate}
\end{cor}
\begin{proof}
It is straightforward consequence of Proposition \ref{prop:loc} and of the fact that in a domain with the $t$-finite character $t= \wedge_{M \in \tmax(R)} t_M$ \cite[Proposition 1.8]{FPT}.
\end{proof}

\begin{cor} With the notation of Proposition \ref{prop:loc}, the following statements are equivalent:
\begin{enumerate}
\item[(i)] $R$ is Clifford $\ast$-regular (respectively, $\ast$-stable);

\item[(ii)] $R_P$ is Clifford $\ast_P$-regular (respectively, $\ast_P$-stable) and $(I (I^\ast:I^2)R_P)^{\ast_P} = (IR_P ((IR_P)^{\ast_P} : (I R_P)^2))^{\ast_P}$, for each ideal $I$ and each $P \in \mc F$.
\end{enumerate}
\end{cor}
\begin{proof} It follows directly from Proposition \ref{prop:loc}, since each ideal of $R_P$ is extended from $R$.
\end{proof}

The following corollary for $\ast=d$  is proved in \cite[Proposition 2.8]{B3}.

\begin{cor} \label{locw2}
Let  $\ast =\tilde{\ast}$ be a star operation on $R$ spectral and of finite type. Then the following conditions are equivalent:

\begin{enumerate}
\item[(i)]  $R$ is Clifford $\ast$-regular (respectively, $\ast$-stable);

\item[(ii)] $R_M$ is Clifford regular (respectively, stable) and $I(I^\ast:I^2)R_M = IR_M(IR_M:I^2)$, for each  ideal $I$ and $M\in \astmax(R)$.
\end{enumerate}
\end{cor}

\begin{remark} \rm Let $\ast=\tilde{\ast}$. If the ideal $I$ is $\dot{\ast}$-finite in $E(I^\ast)$ then $(I^\ast:I^2)R_M = (IR_M:I^2)$ \cite[Lemma 1.8]{GP}. Hence in this case $I$ is Clifford $\ast$-regular (respectively, $\ast$-stable)  if and only if  $IR_M$ is Clifford regular (respectively, stable) for each $M\in \ast$-$\Max(R)$.
However if $R$ is Clifford $\ast$-regular and not $\ast$-stable not every ideal $I$ is   $\dot{\ast}$-finite in $E(I^\ast)$. For example if $V$ is a one-dimensional valuation domain with maximal ideal $M$, $V$ is Clifford regular and it is stable if and only if $V$ is discrete, equivalently $M$ is finitely generated in $V=(M:M)$.
\end{remark}

A family $\mathcal{F}$ of prime ideals of $R$ such that no two primes in $\mathcal{F}$ contain a common nonzero prime ideal  is called \emph{independent}.
If $R$ has  a  defining independent family $\mathcal{F}$ and the intersection $R=\bigcap_{P \in \mathcal F} R_P$ has finite character,  $R$ is called an \emph{$\mathcal{F}$-IFC domain} \cite{AZ}.
When $\mathcal{F} = \Max(R)$ (respectively, $\mathcal{F} = \tmax(R)$), an $\mathcal{F}$-IFC domain is called \emph{$h$-local} (respectively, \emph{weakly Matlis}).

D.D. Anderson and M. Zafrullah have shown that if $R$ is an $\mathcal{F}$-IFC domain and $Y$ is a $\ast_\mathcal{F}$-submodule of $K$, then for each $P \in \mathcal{F}$ and $X$ submodule of $Y$, $(Y:X)R_P = (YR_P : X R_P)$ \cite[Corollary 5.2]{AZ}.

\begin{prop} \label{IFC}
Let $R$ be an $\mathcal{F}$-IFC domain and $\ast:=\ast_\mc F$. The following conditions are equivalent:

\begin{enumerate}
\item[(i)] $R$ is Clifford $\ast$-regular (respectively, $\ast$-stable);

\item[(ii)] $R_P$ is Clifford regular (respectively, stable) for all $P$ in $\mathcal{F}$.
\end{enumerate}
\end{prop}
\begin{proof}
Let $I$ be an ideal of $R$. By Proposition \ref{prop:loc}, we need only show that $IR_P(I^{\ast} :I^2)R_P = IR_P (IR_P : (IR_P)^2)$.  Since $I^{\ast}$ is a $\ast$-ideal and $I^2 \sub I^{\ast}$, it follows by \cite[Corollary 5.2]{AZ} and the definition of $\ast:=\ast_\mc F$ that $(I^{\ast} : I^2)R_P =   (I^{\ast} R_P : I^2 R_P ) = (IR_P : (IR_P)^2)$.
\end{proof}

\begin{cor}
Let $R$ be an $h$-local domain. Then $R$ is Clifford regular (respectively, stable) if and only if it is locally Clifford regular (respectively, stable).
\end{cor}
\begin{proof}
It follows from Proposition \ref{IFC} for  $\mathcal{F} = \Max(R)$.
\end{proof}

\begin{cor} \label{localwM}
Let $R$ be a weakly Matlis domain. Then:

\begin{enumerate}
\item[(1)] $R$ is Clifford $t$-regular (respectively, $t$-stable) if and only if $R_M$ is Clifford $t_M$-regular (respectively, $t_M$-stable) for each $M\in \tmax(R)$.
\item[(2)] $R$ is Clifford $w$-regular (respectively, $w$-stable) if and only if $R_M$ is Clifford regular (respectively, stable) for each $M\in \tmax(R)$.
\end{enumerate}
\end{cor}
\begin{proof}
(1) follows from Corollary \ref{loc-t} and \cite[Corollary 5.2 and Corollary 5.3]{AZ}.
(2) follows from Proposition \ref{IFC}, when $\mathcal{F} = \tmax(R)$.
\end{proof}

\begin{remark} \rm An important  class of weakly Matlis domains is given by the so called $w$-divisorial domains, studied in \cite{GE}.
We recall that a \emph{divisorial domain} is a domain in which each ideal is divisorial and that  a \emph{$w$-divisorial domain} is a domain in which each $w$-ideal is divisorial. It is known that $R$ is a divisorial domain (respectively, a $w$-divisorial domain) if and only if $R$ is $h$-local (respectively, weakly Matlis) and locally (respectively, $t$-locally) divisorial \cite[Theorem 1.5]{GE}.

Hence, when $R$ is ($w$-)divisorial, by Corollary \ref{localwM} $R$ is Clifford ($w$-)regular (equivalently $t$-regular) if and only if $R_M$ is Clifford regular for all $M\in (t$-$)\Max(R)$.
\end{remark}

It is known that, for any star operation $\ast$ of finite type, a Clifford $\ast$-regular domain has the \emph{$\ast$-local $\ast$-invertibility property}, that is an ideal $I$ is $\ast$-invertible if and only if  $I^\ast R_M$  is principal for each $M\in \astmax(R)$ \cite[Lemma 4.4]{HK}. The next consequence of Proposition \ref{prop:loc} extends the property that a Clifford regular domain has the \emph{local stability property}, which means that an ideal $I$ is stable if and only if $IR_M$ is stable for each $M\in \Max(R)$ \cite[Lemma 5.7]{B3}.

\begin{cor} \label{locstable} With the notation of Proposition \ref{prop:loc}, assume that $R$ is Clifford $\ast$-regular. Then an ideal $I$ is $\ast$-stable if and only if $IR_P$ is $\ast_P$-stable for all $P\in \mc F$. Hence  $R$ is $\ast$-stable if and only if $R_P$ is $\ast_P$-stable for all $P\in \mc F$.
\end{cor}

\begin{cor} \label{locw3}
Let  $\ast =\tilde{\ast}$ be a star operation on $R$ spectral and of finite type and assume that $R$ is Clifford $\ast$-regular. Then $R$ is $\ast$-stable if and only if $R_M$ is  stable for all $M\in \astmax(R)$.
\end{cor}


\section{The local case}

If $\ast=\tilde{\ast}$ is spectral and of finite type,  the problem of establishing when a Clifford $\ast$-regular domain is $\ast$-stable can be reduced to the local case  (Corollary \ref{locw3}). Note that any local stable domain is strongly stable \cite[Lemma 3.1]{O2}.

The following useful result is due to Olberding.

\begin{theorem} \label{Olb} \cite[Theorem 2.3]{O1} A domain $R$ is stable if and only if  (a) $R$ is finitely stable; (b) $PR_P$ is a stable ideal of $R_P$, for each nonzero prime $P$; (c) $R_P$ is a valuation domain for each nonzero nonmaximal prime $P$; (d) $R$ has finite character.
\end{theorem}

We recall that any valuation domain $V$ is Clifford regular \cite[Theorem 3]{BS} and that a valuation domain is stable if and only if it is strongly discrete, that is $(PV_P)^2\neq PV_P$ for each nonzero prime ideal $P$ \cite[Proposition 5.3.8]{FHP}.

 \begin{cor} \label{corOlb} Let $R$ be a local finitely stable (in particular Clifford regular) domain with maximal ideal $M$. The following conditions are equivalent:
 \begin{itemize}
\item[(i)]  $R$ is (strongly) stable;
\item[(ii)] $R_P$ is a valuation domain for each nonzero nonmaximal prime ideal $P$ and $PR_P$ is a stable ideal for each nonzero prime ideal $P$;
\item[(iii)] $R_P$ is a strongly discrete valuation domain for each nonzero nonmaximal prime ideal $P$ and $M$ is a stable ideal.
\end{itemize}
 \end{cor}

 \begin{proof} (i) $\lra$ (ii) follows from Theorem \ref{Olb}.

 (i) $\ra$ (iii) For each nonmaximal prime ideal $P$, $R_P$ is a valuation domain (via (ii)) and is stable (Corollary \ref{overrings}). Hence $R_P$ is strongly discrete \cite[Proposition 5.3.8]{FHP} and clearly $M$ is stable.

 (iii) $\ra$ (ii) because strongly discrete valuation domains are stable \cite[Proposition 4.1]{O3}.
 \end{proof}

 The problem of deciding when a local one-dimensional Clifford regular domain is stable was investigated by  Zanardo in \cite{Z}.  Next theorem adds new conditions to  \cite[Theorem 2.12]{Z}.

 \begin{theorem} \label{onedimstable} Let $R$ be a local one-dimensional domain with maximal ideal $M$ and let $E:=(M:M)$. The following statements are equivalent:
\begin{itemize}
\item[(i)]  $R$ is (strongly) stable;
\item[(ii)] $R$ is Clifford regular and $M$ is a stable ideal;
\item[(iii)]  $R$ is Clifford regular and each maximal ideal of $E$ is divisorial;
\item[(iv)] $R$ is Clifford regular and each $t_E$-maximal ideal of $E$ is divisorial.
\end{itemize}
\end{theorem}
\begin{proof} (i) $\lra$ (ii) by Proposition \ref{corOlb}.

(i) $\ra$ (iii) If $R$ is stable, by Proposition \ref{prop1}(1) $R$ is Clifford regular. Since overrings of stable domains are stable \cite[Theorem 5.1]{O2}, $E:=(M:M)$ is a stable overring of $R$. Hence each maximal ideal of $E$ is divisorial by \cite[Proposition 1.5]{GP}.

(iii) $\ra$ (ii)  By Proposition \ref{prop1}(2), $M$ is $v_E$-invertible in $E$. Since each maximal ideal of $E$ is divisorial, it follows that $M$ is invertible in $E$.

(iii) $\lra$ (iv) If $R$ is Clifford regular, then $R$ is finitely stable \cite[Proposition 2.3]{B3}.  Hence $E:=(M:M)$ is an integral extension of $R$ \cite[Lemma 4.1]{O2} and so $E$ is one-dimensional. It follows that $\Max(E)=\tmax(E)$.
\end{proof}

Recall that a \emph{Mori domain} is a domain satisfying the ascending chain condition on divisorial ideals.  This condition implies that the $v$-operation  and the $t$-operation coincide.
By the previous theorem, we get that Clifford regularity and stability coincide for one-dimensional Mori domains.  The study of star regularity and star stability for Mori domains is deepened in \cite{GP2}.

\begin{cor} \label{Moristable} Let $R$ be a local one-dimensional Mori domain.
The following conditions are equivalent:
\begin{itemize}
\item[(i)] $R$ is Clifford regular;
\item[(ii)] $R$ is (strongly) stable;
\item[(iii)] $R$ is Boole regular.
\end{itemize}
\end{cor}
\begin{proof} (i) $\ra$ (ii) Let $M$ be the maximal ideal of $R$. Since $R$ is Mori, then $E:=(M:M)$ is Mori, because $M$ being of height one is a $t$-ideal (\cite[Proposition 4.3]{CCI} and \cite[Proposition 6.3]{CCII}).
Hence the $t$- and the $v$-operation coincide on $E$ and we can apply Theorem \ref{onedimstable}.

(ii) $\ra$ (iii) $\ra$ (i) are clear.
\end{proof}

We now consider \emph{pseudo-valuation domains}. Let $V$ be  a valuation domain with maximal ideal $M$ and residue field $K$ and let $k$ be a subfield of $K$. Let $R:=\pi^{-1}(k)$ be the pseudo-valuation domain arising from the following pullback diagram:

$$  \begin{CD}
        R   @>>>    k\\
        @VVV        @VVV    \\
        V  @>\pi>>   K\\
        \end{CD}$$

 (where $\pi$ is the canonical surjection).

  \begin{prop} \label{PVD1} With the notation above, assume that $R\neq V$. The following conditions are equivalent:
\begin{itemize}
\item[(i)] $R$ is Clifford regular;
\item[(ii)] $R$ is divisorial;
\item[(iii)] $[K:k] = 2$.
\end{itemize}
\end{prop}
\begin{proof} (i) $\lra$ (iii) is \cite[Theorem 5.1(2)]{KM1} and (ii) $\lra$ (iii) is in \cite[Corollary 3.5]{M2}.
\end{proof}

  In the one-dimensional case, the next result can be easily deduced from Theorem \ref{onedimstable}.

   \begin{prop} With the notation above,  assume that $R\neq V$. The following conditions are equivalent:
\begin{itemize}
\item[(i)] $R$ is (strongly) stable;
\item[(ii)] $R$ is Boole regular and $V$ is strongly discrete;
\item[(iii)] $R$ is Clifford regular and $V$ is strongly discrete;
\item[(iv)] $V$ is strongly discrete and $[K:k]= 2$;
\item[(v)] $R$ is totally divisorial (i.e., $R$ and all of its overrings are divisorial).

\end{itemize}
\end{prop}
\begin{proof} (i) $\ra$ (ii)  An overring of a stable domain is stable \cite[Theorem 5.1]{O2} and stable valuation  domains are strongly discrete \cite[Proposition 5.3.8]{FHP}, hence $V$ is strongly discrete. Moreover $R$ is Boole regular by Proposition \ref{prop1}(3).

(ii) $\ra$ (iii) is clear and (iii) $\ra$ (iv) follows from \cite[Theorem 5.1]{KM1}.

(iv) $\ra$ (v)
Since $[K:k]=2$,  $R$ is divisorial by \cite[Corollary 3.5]{M2}.  Moreover, this is true for all the overrings of $R$. Indeed the only overrings of $R$ are $V$ and its localizations, which are divisorial since $V$ is strongly discrete \cite[Proposition 4.1]{O3}.

(v) $\ra$ (i) A totally divisorial domain is stable by \cite[Theorem 3.12]{O5}.
\end{proof}

 \section{The $\ast$-integral closure}

Given a semistar operation $\ast$ on $R$, the \emph{$\ast$-integral
closure} of $R$ is the integrally closed overring of $R$ defined by
$R^{[\ast]}:= \bigcup \{(J^\ast:J^\ast); \, J\in\ol{\mc F}(R) \mbox{ finitely generated}\}$ \cite{fl01}. Clearly $R^{[\ast]}=R^{[\ast_f]}$.
We say that $R$ is \emph{$\ast$-integrally closed} if $R=R^{[\ast]}$. In this case it easy to see that $\ast$ is necessarily a (semi)star operation on $R$.

When $\ast=d$ is the identity, we obtain the integral closure of $R$, here denoted by $R^\prime$. When $\ast=v$,
$R^{[v]}= R^{[t]}$ is called the \emph{pseudo-integral closure of $R$} \cite{AHZ}.
We have $R\sub R^\prime\sub R^{[\ast]}$. In addition, if  $\widetilde{R}:= \bigcup \{(I^v:I^v); \, I\in \ol{\mc F}(R) \}$ is the \emph{complete integral closure} of $R$, we also have $R^{[\ast]}\sub R^{[v]}\sub \widetilde{R}$.

The $w$-integral closure has been widely studied \cite{W, CZ} and many of its properties have been extended to integral closures with respect to semistar operations spectral and of finite type. For example, if $\ast=\tilde{\ast}$,
 $R$ is $\ast$-integrally closed (i.e., $R=R^{[\ast]}$) if and
only if $R$ is integrally closed (i.e., $R=R^\prime$) \cite[Lemma 4.13]{EFP}. This shows that, if $\ast^\prime$ is any semistar operation  spectral and of finite type on $R^{[\ast]}$, we have $(R^{[\ast]})^{[\ast^\prime]}=R^{[\ast]}$.

In addition, $R^{[\ast]}$ satisfies $\ast$LO, $\ast$GU and $\ast$INC \cite[Lemma2.15(b)]{CF}.

\begin{theorem} \label{*LO}
Let  $\ast = \tilde{\ast}$ be a semistar operation spectral and of finite type on $R$ and let $D$ be an overring of $R$ such that $D\sub R^{[\ast]}$. Then:
\begin{itemize}
\item[(a)] {\rm($\ast$LO)} For each $\ast$-prime ideal $P$ of $R$, there is a $\ast_{\vert_D}$-prime ideal $Q$ of $D$ such that $Q\cap R=P$.
\item[(b)] {\rm ($\ast$GU)} Given two $\ast$-primes $P_1 \subsetneq P_2$ in $R$ and a $\ast_{\vert_D}$-prime $Q_1$  in $D$ such that $Q_1 \cap R = P_1$, there exists a  $\ast_{\vert_D}$-prime $Q_2$ in $D$ satisfying $Q_1 \subsetneq Q_2$ and $Q_2 \cap R=P_2$.
\item[(c)] {\rm($\ast$INC)} Two different $\ast_{\vert_D}$-primes of $D$ with the same contraction in $R$ cannot be comparable.
\end{itemize}
\end{theorem}
\begin{proof} For $D=R^{[\ast]}$, this was proved in \cite[Lemma 2.15(b)]{CF}.

Now let $D$ be such that $R\sub D\sub R^{[\ast]}$ and let $P$ be a $\ast$-prime of $R$. Set $\dot{\ast}:=\ast_{\vert{R^{[\ast]}}}$  and let $Q$ be a $\dot{\ast}$-ideal of $R^{[\ast]}$ lying over $P$.  If $Q^\prime := Q \cap D$, clearly $Q^\prime$ lies over $P$ and $(Q^\prime)^{\ast_{\vert_D}} \neq D$, since $(Q^\prime)^{\ast_{\vert_D}} = (Q \cap D)^\ast \subseteq Q^\ast$ and $1 \not \in Q^\ast = Q^{\ast_{\vert_D}}$. This proves the $\ast$-LO for the pair $R,D$.

The $\ast$-GU is similar and $\ast$-INC is a consequence of $\ast$-INC for the pair $R, R^{[\ast]}$ and  $\ast_{\vert_D}$-GU for the pair $D, R^{[\ast]}$.
\end{proof}

Recall that, when $\ast = \tilde{\ast}$, the \emph{$\ast$-dimension} of $R$ (denoted by $\ast$-$\dim(R)$) is the supremum of the heights of the $\tilde{\ast}$-maximal ideals.

\begin{cor}  \label{lemmatdimone} Let  $\ast = \tilde{\ast}$ be a semistar operation spectral and of finite type on $R$ and let $D$ be an overring of $R$  such that $D\sub R^{[\ast]}$. Then
$\ast$-$\dim(R) = \ast_{\vert_D}$-$\dim(D)$. In particular, $R$ has $\ast$-dimension one if and only if $R^{[\ast]}$ has $\ast_{\vert R^{[\ast]}}$-dimension one.
 \end{cor}

The integral closure of a Clifford regular domain is a Pr\"ufer domain \cite[Proposition 2.3]{KM1}: in fact a Clifford regular domain is finitely stable \cite[Proposition 2.3]{B3} and a finitely stable domain has  Pr\"ufer  integral closure \cite[Proposition 2.1]{Rush}. In addition, the integral closure of a Boole regular domain is Bezout \cite[Proposition 2.3]{KM1}.
When $\ast=\tilde{\ast}$, we extend these results to $\ast$-regularity; for $\ast$-stability see \cite[Theorem 2.3]{GP}.

 For a
star operation $\ast$, $R$ is called a \emph{Pr\"ufer $\ast$-multiplication domain} (for short a P$\ast$MD) if every nonzero finitely generated ideal of $R$ is $\ast_f$-invertible. It is clear that $R$ is a P$\ast$MD if and only if it is a P$\ast_f$MD; moreover, if $R$ is a P$\ast$MD, it is integrally closed (in fact, it is a P$v$MD) and  $\ast_f=\tilde{\ast}$ \cite[Theorem 3.1]{fjs}. Also recall that $R$ is called a  GCD-\emph{domain} if any two nonzero elements of $R$ have a greatest common divisor. A GCD-domain is a P$v$MD; more precisely $R$ is GCD-domain if and only if each $t$-finite ideal is principal.

\begin{theorem} \label{wic}  Let $\ast = \tilde{\ast}$ be a star operation spectral and of finite type.    If $R$ is Clifford $\ast$-regular, then $D:=R^{[\ast]}$ is a P$\dot{\ast}$MD and $\dot{\ast}=w_D= t_D$.  In addition, if $R$ is Boole $\ast$-regular,  $D$ is a GCD-domain. \end{theorem}
 \begin{proof}
Assume that $R$ is Clifford $\ast$-regular. We have that $\ast = w$, by Corollary \ref{ast=w}. Since $D:= R^{[w]}$ is $t$-linked \cite[Lemma 1.2]{CZ}, $\dot{w}$ is a star operation on $R^{[w]}$.
  We have to prove that each finitely generated ideal $I$ of $D$ is $\dot{w}$-invertible. Since $D$ is contained in the quotient field of $R$, $I$  is an ideal of $R$.
Since $R$ is Clifford $w$-regular, $I$ is $w$-stable by Proposition \ref{fingen}. Thus $I$ is $w_{\vert_E}$-invertible in $E(I^w)$. But $E(I^w) = E(I^{\dot{w}}) = D$, since we have remarked that the $w$-integral closure is always $\ast$-integrally closed for all star operations spectral and of finite type. Thus $I$ is $\dot{w}$-invertible in $D$.

Now assume that $I$ is Boole $w$-regular. Since $I$ is $w$-stable, then $I^w$ is principal in $E = D$ (Proposition \ref{prop1}(3)). Since $\dot{w}= t_D$, it follows that $D$ is a GCD-domain.
 \end{proof}

Bazzoni showed that an integrally closed Clifford regular domain is precisely a Pr\"ufer domain with finite character \cite[Theorem 4.5]{B3}. Later Kabbaj and Mimouni proved that a P$v$MD is Clifford $t$-regular if and only if it has $t$-finite character \cite[Theorem 3.2]{KM2} and gave an example of an integrally closed Clifford $t$-regular domains that is not  a P$v$MD \cite[Example 2.8]{KM2}.
However they conjectured that a pseudo-integrally closed Clifford $t$-regular domain $R$ be a P$v$MD \cite[Conjecture 3.1]{KM2} and gave a positive answer   for strongly $t$-discrete domains (domains such that $P\neq (P^2)^t$ for each $t$-prime ideal $P$) \cite[Corollary 3.12]{KM2}.
Finally Halter-Koch, in the language of ideal systems, proved that, for any star operation $\ast$ of finite type, a P$\ast$MD is Clifford $\ast$-regular if and only if it has $\ast$-finite character \cite[Theorems 6.3 and 6.11]{HK} and solved in positive the conjecture of Kabbaj and Mimouni \cite[Proposition 6.12]{HK}.

 \begin{theorem} \label{pvmd}  Let $\ast$ be a star operation  of finite type on a domain $R$.
 \begin{itemize}
 \item[(1)] If $R$ is $\ast$-integrally closed and Clifford $\ast$-regular, then $R$ is a P$v$MD.

 \item[(2)] The following conditions are equivalent:
  \begin{itemize}
  \item [(i)]  $R$ is $\tilde{\ast}$-integrally closed and Clifford $\tilde{\ast}$-regular;
   \item[(ii)] $R$ is integrally closed and Clifford $\tilde{\ast}$-regular;
  \item[(iii)] $R$ is a P$\ast$MD of $\ast$-finite character.
\end{itemize}
Under (any one of) these conditions $\tilde{\ast}=w$.
\end{itemize}
\end{theorem}
\begin{proof} (1)  Assume that $R$ is $\ast$-integrally closed and let $I$ be a finitely generated ideal of $R$. Then $(I^\ast:I^\ast)=R$ and $I$ is $t$-invertible in $R$  by Proposition \ref{prop1}(2). Thus $R$ is a P$v$MD.

(2) (i) $\ra$ (iii) $R$ is a P$\ast$MD by Theorem \ref{wic}, thus it has $\ast$-finite character by \cite[Theorem 6.11]{HK}.

(iii) $\ra$ (i) is \cite[Theorem 6.3]{HK}.

(i) $\lra$ (ii)  because a domain is integrally closed if and only if it is $\tilde{\ast}$-integrally closed \cite[Lemma 4.13]{EFP}.

To finish, if $R$ is  Clifford $\tilde{\ast}$-regular domain, $\tilde{\ast}=w$ by  Corollary \ref{ast=w}.
\end{proof}

 \begin{cor} \label{corpvmd}  The following conditions are equivalent:
  \begin{itemize}
 \item[(i)] $R$ is pseudo-integrally closed and Clifford $t$-regular;
  \item [(ii)]  $R$ is $w$-integrally closed and Clifford $w$-regular;
   \item[(iii)] $R$ is integrally closed and Clifford $w$-regular;
  \item[(iv)] $R$ is a P$v$MD of $t$-finite character.
\end{itemize}
\end{cor}

We say that a domain $R$ is \emph{strongly $\ast$-discrete} if $P\neq (P^2)^\ast$ for  each $\ast$-prime ideal $P$. Recalling that $R$ is a P$\ast$MD if and only if $R_M$ is a valuation domain for each $M\in \astmax(R)$ \cite[Theorem 3.1]{fjs}, it is easy to check that $R$ is a strongly $\ast$-discrete P$\ast$MD if and only if $R_M$ is a strongly discrete valuation domain for each $M\in \astmax(R)$.

The next theorem implies that in the integrally closed case, for $\ast$ of finite type,  a Clifford $\ast$-regular domain is $\ast$-stable if and only if it is  strongly $\ast$-discrete. For the identity this was proved in \cite[Lemma 3.4]{KM1}.

  \begin{theorem} \label{stablepvmd} Let $\ast$ be a star operation of finite type on a domain $R$. The following statements are equivalent:
 \begin{enumerate}
 \item[(i)] $R$ is $\ast$-integrally  closed and  $\ast$-stable;
\item[(ii)]  $R$ is  $\tilde{\ast}$-integrally closed and $\tilde{\ast}$-stable;
 \item[(iii)]   $R$ is  integrally closed and $\tilde{\ast}$-stable;
  \item [(iv)] $R$ is a $\ast$-stable P$\ast$MD;
   \item[(v)] $R$ is a strongly $\ast$-discrete P$\ast$MD of $\ast$-finite character;
   \item[(vi)]  $R$ is integrally closed, strongly $\ast$-discrete and $\tilde{\ast}$-regular. \end{enumerate}
 Under (any one of) these conditions $\tilde{\ast}=w$.
  \end{theorem}
 \begin{proof}
 (i) $\lra$ (iv) $\lra$ (ii) by \cite[Corollary 2.4]{GP}.

  (ii) $\lra$ (iii) by \cite[Lemma 4.13]{EFP}.

(iv) $\ra$ (v) Since $\ast=\tilde{\ast}$, $R$ has $\ast$-finite character and $R_M$ is a stable valuation domain for each $M\in \astmax(R)$ \cite[Theorem 1.9]{GP}.
Hence $R_M$ is strongly discrete \cite[Proposition 4.1]{O3}.

(v) $\ra$ (iv) Since a strongly discrete valuation domain is stable \cite[Proposition 4.1]{O3}, $R$ is $\ast$-locally stable with $\ast$-finite character. Since $\ast=\tilde{\ast}$, $R$ is $\ast$-stable by \cite[Theorem 1.9]{GP}.

(v) $\lra$ (vi) by Theorem \ref{pvmd}(2).

\smallskip
To conclude, if $R$ is $\ast$-stable, then $\tilde{\ast}=w$ by  \cite[Corollary 1.6]{GP}.
\end{proof}

\begin{cor} The following statements are equivalent for a domain $R$:
 \begin{enumerate}
 \item[(i)] $R$ is completely integrally closed and  $w$-stable;
\item[(ii)]  $R$ is a Krull domain;
 \item[(iii)]   $R$ is  completely integrally closed strongly $t$-discrete and  and $w$-regular.
 \end{enumerate}
\end{cor}
 \begin{proof}
 (i) $\lra$ (ii) is \cite[Corollary 2.5]{GP}.

 (i) $\lra$ (iii) by Theorem  \ref{stablepvmd}.
 \end{proof}

 We now show that, for $\ast=\tilde{\ast}$, in $\ast$-dimension one $\ast$-regularity and $\ast$-stability are equivalent if and only if $R^{[\ast]}$ is a Krull domain.

\begin{theorem} \label{tdimone} Let $\ast=\tilde{\ast}$ be a star operation on $R$ spectral and of finite type. The following conditions are equivalent:
  \begin{itemize}
  \item [(i)] $R$ is Clifford $\ast$-regular and $R^{[\ast]}$ is a Krull domain;
  \item[(ii)] $R$ is $\ast$-stable of $\ast$-dimension one.
\end{itemize}
Under (any one of) these conditions, $\ast=w$ and $R^{[\ast]}=R^{[w]}=\widetilde{R}$ is the complete integral closure of $R$.
\end{theorem}
\begin{proof} Since a Krull domain is completely integrally closed, if  $R^{[\ast]}$ is  Krull, we have $\widetilde{R^{[\ast]}}=R^{[\ast]}$. Hence, from $R\sub R^{[\ast]}\sub\widetilde{R}$, we obtain $\widetilde{R}\sub \widetilde{R^{[\ast]}}= R^{[\ast]}\sub \widetilde{R}$ and $R^{[\ast]}= \widetilde{R}$.

(i) $\ra$ (ii)
Assume that $R$ is Clifford $\ast$-regular. Given a nonzero ideal $I\sub R$, set  $E:=E(I^\ast)$ and $T:= T(I^\ast)$. We have to prove that $T^\ast=E$.

Set $D:=R^{[\ast]}$ and denote by $\ast_{\vert_E}$ and $\ast_{\vert_D}$ the operations induced respectively on $E$ and $D$ by the star operation $\ast$ of $R$. Note that $\ast_{\vert_D}=t_D$ is the $t$-operation on $D$ by Theorem \ref{wic}.

By Lemma \ref{lemmaX}, we have $T^\ast=(T^2)^\ast$.  Whence, $(TD)^{t_D}=(T^2D)^{t_D}$.
Since $D$ is Krull, $TD$ is $t_D$-invertible and so $(TD)^{t_D}=D$.

Now  $R\subseteq E \subseteq  \widetilde{D}=D$ implies $D:=R^{[\ast]} \subseteq E^{[\ast_{\vert_E}]} \subseteq D^{[\ast\vert_D]} = D$. Hence $E^{[\ast_{\vert_E}]} = D$.

If $T^\ast \neq E$, there exists a $\ast_{\vert_E}$-maximal ideal of $E$ such that $T \subseteq N$. Now, since $E^{[\ast_{\vert_E}]} = D$, by Proposition \ref{*LO}, there exists  a  $\ast_{\vert_D}$-prime ideal $M$ of $D$ lying over $N$. Thus $T \subseteq N \subseteq M$ implies $(TD)^{\ast_{\vert_D}} \subseteq M ^{ \ast_{\vert_D}}= M \subsetneq D$, a contradiction. Hence $T^\ast = E$.

To finish, since $\ast_{\vert_D}=t_D$ on $D$ and since $t_D$-$\dim(D)=1$,  then $R$ has $\ast$-dimension one  by Corollary \ref{lemmatdimone}.

(ii) $\ra$ (i) If $R$ is $\ast$-stable, $\ast=w$ by \cite[Corollary 1.6]{GP}.  Now recall that if $R$ is $w$-stable, $D:=R^{[w]}$ is a $w_D$-stable overring of $R$ \cite[Corollary 2.6]{GP}, hence $R$ is a strongly $t$-discrete P$v$MD with $t$-finite character \cite[Theorem 2.9]{GP}. Since $D$ has $\dot{\ast}$-dimension one,  by Corollary \ref{lemmatdimone},  $D$ has $t$-dimension one and so it is a Krull domain. Finally $\ast$-stability implies $\ast$-regularity by Proposition \ref{prop1}(1).
\end{proof}

When $\ast=d$ is the identity, we obtain the following corollary.

\begin{cor} The following conditions are equivalent:
  \begin{itemize}
  \item [(i)] $R$ is Clifford regular and $R^\prime$ is a Dedekind domain;
  \item[(ii)] $R$ is stable of dimension one.
\end{itemize}
\end{cor}


\section{The $\ast$-finite character}

 Olberding proved that a domain $R$ is stable if and only if it is locally stable and has finite character \cite[Theorem 3.3]{O2}. This result, for a star operation $\ast=\tilde{\ast}$ spectral and of finite type,  was extended by the authors to star stability in \cite[Theorem 1.9]{GP}.
The finite character of Clifford regular domains was proved by Bazzoni in \cite[Theorem 4.7]{B4}. With a more direct argument, in the next theorem we prove that, for $\ast=\tilde{\ast}$, Clifford  $\ast$-regular domains have the $\ast$-finite character.

We will use the fact that, for $\ast$ of finite type,  a Clifford $\ast$-regular domain has the $\ast$-local $\ast$-invertibility property  (i.e., each ideal $I$ such that $I^\ast R_M$ is principal for all $M\in \astmax(R)$ is $\ast$-invertible) \cite[Lemma 4.4]{HK} and  the following results from \cite{FPT} (or from \cite{ZD}).

\begin{theorem} \label{teoFPT} Let $\ast$ be a star operation of finite type on $R$. Then:
\begin{enumerate}
\item[(1)]  (\cite[Proposition 1.6]{FPT} or \cite[Corollary 4]{ZD}) $R$ has the $\ast$-finite character if and only if, for each nonzero $x \in R$, any family of pairwise $\ast$-comaximal $\ast$-finite $\ast$-ideals containing $x$ is finite.
\item[(2)]Ê\cite[Proposition 2.1]{FPT}  If $R$ has the $\ast$-local $\ast$-invertibility property, for each nonzero $x \in R$, any family of pairwise $\ast$-comaximal $\ast$-invertible $\ast$-ideals containing $x$ is finite.
\end{enumerate}
\end{theorem}

\begin{theorem} \label{FC}  Let $\ast=\tilde{\ast}$ be a star operation on $R$ spectral and of finite type.
If $R$ is Clifford $\ast$-regular, then $R$ has the $\ast$-finite character.
\end{theorem}
\begin{proof} Assume that $R$ does not have the $\ast$-finite character. Then there exists a  nonzero element $x\in R$ which is contained in infinitely many pairwise $\ast$-comaximal $\ast$-finite $\ast$-ideals $I_\al$ (Theorem \ref{teoFPT}(1)). Set $E_\al:=(I_\al:I_\al)$ and consider the $R$-module $D:=\sum_\al E_\al$. We claim that $E:=D^\ast$ is a fractional overring of $R$. In fact, since $x\in I_\al$ for each $\al$, $xD\sub \sum_\al I_\al\sub R$ and so $D$ is a fractional ideal of $R$. In addition, $E$ is closed under multiplication. Indeed,
let $a_1$, $a_2\in E$. Since $\ast$ is of finite type, there are two finitely generated ideals $F_1$ and $F_2$ contained in $D$ such that $a_1\in F_1^\ast$ and $a_2\in F_2^\ast$. So we can assume that $a_1$, $a_2\in G:=(E_{\al_1}+\dots+E_{\al_n})^\ast$, for some distinct overrings $E_{\al_1},\dots, E_{\al_n}$ of $R$.  We want to show that $(G^2)^\ast=(\sum E_{\al_i}E_{\al_j})^\ast\sub E$. For this, we note that $E_{\al_i}^2\sub E_{\al_i}$ while, for $i\neq j$,  $E_{\al_i} E_{\al_j}\sub (E_{\al_i}+E_{\al_j})^\ast$.
 In fact, for any two fixed indexes $\be$ and $\gamma$, we have $E_{\be} E_{\gamma}\sub (I_\be I_{\gamma}:I_{\be} I_{\gamma})$.
  Now,  let $a\in K$ be such that $aI_\beta I_\gamma\sub I_\beta I_\gamma$. Then $aI_\beta\sub (I_\beta I_\gamma:I_\gamma)$ and $aI_\gamma\sub (I_\beta I_\gamma:I_\beta)$. Since, by $\ast$-comaximality, $(I_\beta + I_\gamma)^\ast=R$, we conclude that $$aR=a(I_\beta + I_\gamma)^\ast \sub ((I_\beta I_\gamma:I_\beta)+(I_\beta I_\gamma:I_\gamma))^\ast\sub (E_\be+E_\gamma)^\ast.$$

Since $E^\ast=E$, the restriction of $\ast$ to $E$, denoted as usual as $\dot{\ast}$, is a star operation on $E$ and clearly the extension $R\sub E$ is $(\ast, \dot{\ast})$-compatible.  It follows that  $E$, being a fractional overring of $R$, is a Clifford $\dot{\ast}$-regular domain (Proposition \ref{overrings1}) and thus it has the $\dot{\ast}$-local $\dot{\ast}$-invertibility property \cite[Lemma 4.4]{HK}. We claim that $\{(I_\al E)^{\dot{\ast}}\}$ is a family of pairwise $\dot{\ast}$-comaximal $\dot{\ast}$-invertible $\dot{\ast}$-ideals of $E$ containing $x$,  getting a contradiction with Theorem \ref{teoFPT}(2).

First of all, by Proposition \ref{fingen}, each ideal $I_\al$ is $\ast$-stable, hence $I_\al E$ is $\dot{\ast}$-invertible in $E$.
Then observe that $E\sub R^{[\ast]}$, thus by Theorem \ref{*LO}
each $\ast$-maximal ideal of $R$ is contained in a $\dot{\ast}$-prime ideal of $E$. It follows that $(I_\al E)^{\dot{\ast}}\neq E$ for each $\al$.  In addition, for $\al\neq \be$ the ideals $I_\al E$ and $I_\be E$ are $\dot{\ast}$-comaximal, since the contraction of a $\dot{\ast}$-prime ideal of $E$ is a $\ast$-prime of $R$.
\end{proof}

The next result for the identity was proven in \cite[Corollary 4.8]{B4}.

\begin{prop} Let $\ast=\tilde{\ast}$ be a star operation on $R$ spectral and of finite type.
If $R$ is Clifford $\ast$-regular, its $\ast$-integral closure  $R^{[\ast]}$ is Clifford $\dot{\ast}$-regular.
\end{prop}
\begin{proof} We can assume $\ast = w$ (Corollary \ref{ast=w}).
If $R$ is Clifford $w$-regular, $D:= R^{[w]}$ is a P$\dot{w}$MD (Theorem \ref{wic}). By Theorem \ref{pvmd}, it is enough to show that $D$ has the $\dot{w}$-finite character.

Let $M \in \tmax(R)$. Note that $D:= R^{[w]} \subseteq R_M^{[w_{\vert R_M}]} = (R_M)^\prime$ (since $w_{\vert R_M} = d_{R_M}$). So, $D_{R \setminus M} \subseteq (R_M)^\prime$. Conversely, $R_M \subseteq D_{R \setminus M}$. Since $D_{R \setminus M}$ is integrally closed, we  obtain $(R_M)^\prime \subseteq D_{R \setminus M}$ and so $(R_M)^\prime = D_{R \setminus M}$. Now $R_M$ is finitely stable (being Clifford regular).  Hence  $(R_M)^\prime = D_{R \setminus M}$ is a Pr\"ufer domain with finitely many maximal ideals (\cite[Proposition 2.1]{Rush} and \cite[Corollary 2.5]{O2}). Since the $\dot{w}$-maximal ideals of $D$ lying over  $M$ are exactly the contractions in $D$ of the maximal ideals of $D_{R \setminus M} (= (R_M)^\prime)$, they are finitely many. Since $R$ has $w$-finite character (Theorem \ref{FC}), we conclude that $R^{[w]}$ has the $\dot{w}$-finite character.
\end{proof}

 In general it is not known whether, when $\ast=\tilde{\ast}$, a domain with $\ast$-finite character that is $\ast$-locally Clifford regular (and not $\ast$-stable) is indeed Clifford $\ast$-regular. By using Theorem \ref{pvmd}, it is easy to see that this is true if $R$ is integrally closed.

 \begin{prop} Let $\ast=\tilde{\ast}$ be a star operation spectral and of finite type on $R$. If $R$ is integrally closed, the following conditions are equivalent:
 \begin{enumerate}
\item[(i)] $R$  is Clifford $\ast$-regular;
\item[(ii)]   $R_M$ is Clifford regular for each $M\in \astmax(R)$ and $R$ has the $\ast$-finite character.
\end{enumerate}
 \end{prop}
 \begin{proof} (i) $\ra$ (ii) follows from Corollary \ref{locw2} and Theorem \ref{FC}.

 (ii) $\ra$ (i) Since a local integrally closed Clifford regular domain is a valuation domain \cite[Theorem 3]{BS}, $R$ is a P$\ast$MD with $\ast$-finite character. Hence $R$ is Clifford $\ast$-regular by Theorem \ref{pvmd}.
 \end{proof}

 We now show that a similar result is true for domains of $\ast$-dimension one.

Given an independent family $\mathcal{F}$ of prime ideals, if $\ast:=\ast_\mathcal{F}$ is of finite type, the finite character of the intersection $R:=\bigcap\{R_P\,;\; P \in \mc F\}$ is equivalent to the property that  each ideal $I$ such that $IR_P$ is finitely generated for all $P\in \mathcal{F}$ is $\ast$-finite \cite[Theorem 3.3]{AZ}. We next observe that this property is equivalent to the  $\ast$-local $\ast$-invertibility property.

\begin{prop} \label{IFC2} Let $\ast=\tilde{\ast}$ be a star operation spectral and of finite type and assume that each $\ast$-prime ideal of $R$ is contained in a unique $\ast$-maximal ideal (e.g., $R$ has $\ast$-dimension one).  Then the following conditions are equivalent:
\begin{enumerate}
\item[(i)] $R$  has the $\ast$-finite character;
\item[(ii)]   $R$  has the $\ast$-local $\ast$-invertibility property.
\end{enumerate}
\end{prop}
\begin{proof}
(i) $\ra$ (ii) By \cite[Theorem 3.3 (1 $\ra$ 5)]{AZ}, if  $IR_M$ is principal for all $M\in \astmax(R)$ then $I$ is $\ast$-finite. Since $\ast$ is of finite type, this implies that $I$ is $\ast$-invertible \cite[Theorem 2.6]{K}.

(ii) $\ra$ (i) Let $M\in \astmax(R)$. Since $\astmax(R)$ is independent, if $x\in M$ is a nonzero element, $M$ is the only $\ast$-maximal ideal containing the ideal $I:=xR_M\cap R$ \cite[Lemma 2.3]{AZ}. Since $IR_M=xR_M$ and $IR_N=R_N$ for all $N\in \astmax(R)$ and $N\neq M$, then $I$ is $\ast$-invertible by hypothesis. By \cite[Theorem 3.3 (4 $\ra$ 1)]{AZ} we conclude that $R$ has $\ast$-finite character.
\end{proof}

\begin{theorem} \label{teoIFC2} Let $\ast=\tilde{\ast}$ be a star operation spectral and of finite type and assume that each $\ast$-prime ideal of $R$ is contained in a unique $\ast$-maximal ideal (e.g.,  $R$ has $\ast$-dimension one).  Then the following conditions are equivalent:
\begin{enumerate}
\item[(i)] $R$ is Clifford $\ast$-regular;
\item[(ii)] $R_M$ is Clifford regular for each $M\in \astmax(R)$ and $R$ has the $\ast$-finite character;
\item[(iii)] $R_M$ is Clifford regular for each $M\in \astmax(R)$ and $R$ has the $\ast$-local $\ast$-invertibility property.\end{enumerate}
Under (any one of) these conditions $\tilde{\ast}=w$.
\end{theorem}
\begin{proof} (i) $\ra$ (ii) $R_M$ is Clifford regular for each $M\in \astmax(R)$ by Corollary \ref{locw2} and $R$ has $\ast$-finite character by Theorem \ref{FC}.

(ii) $\ra$ (i) If the intersection $R$ has the $\ast$-finite character, taking $\mc F:= \astmax(R)$, $R$ is an $\mathcal{F}$-IFC domain.  Hence we can apply Proposition \ref{IFC}.

(ii) $\lra$ (iii) by Proposition \ref{IFC2}.
\end{proof}

Clearly the previous theorem holds for the identity. We now show that it holds also for the $t$-operation.

\begin{prop} \label{wM} Assume that each $t$-prime (respectively, prime) ideal of $R$ is contained in a unique $t$-maximal (respectively, maximal) ideal (e.g., if $R$ has $t$-dimension (respectively, dimension) one). Then the following conditions are equivalent:
\begin{enumerate}
\item[(i)] $R$ has $t$-finite character (respectively, $R$ has finite character);
\item[(ii)] $R$ is weakly Matlis (respectively, $R$ is $h$-local);
\item[(iii)] Each ideal $I$ such that $IR_M$ is $t_M$-finite for each $M\in \tmax(R)$ is $t$-finite (respectively, each ideal that is locally finitely generated  is finitely generated).
\item[(iv)]  $R$ has the $t$-local $t$-invertibility property (respectively, $R$ has the local invertibility property).
\end{enumerate}
\end{prop}
\begin{proof} (i) $\lra$ (ii) is clear. (i) $\ra$ (iii) for the $t$-operation is in \cite[Theorem 1.9]{FPT} and for the identity in \cite[Corollary 3.4]{AZ}.

(iii) $\ra$ (iv) Since $(I^tR_M)^{t_M}=(IR_M)^{t_M}$, if $I^tR_M$ is principal, $IR_M$ is $t_M$-finite. Hence $I$ is $t$-finite and so $t$-invertible. For the identity it is even simpler.

(iv) $\ra$ (ii) Assume that $IR_M$ is principal,  for each $M\in \tmax(R)=\wmax(R)$. Then $IR_M=(IR_M)^{t_M}\sub I^tR_M\sub (I^tR_M)^{t_M}$ and so
$I^tR_M=IR_M$ is principal. By the hypothesis, it follows that $I$ is $t$-invertible, equivalently $w$-invertible.  Hence $R$ has the $w$-local $w$-invertibility property and than it has $t$-finite character by Proposition \ref{IFC2}.  For the identity we can apply directly Proposition \ref{IFC2}
\end{proof}

\begin{cor} \label{teowM2} Assume that each $t$-prime ideal of $R$ is contained in a unique $t$-maximal ideal (e.g., $R$ has $t$-dimension one). Then the following conditions are equivalent:
\begin{enumerate}
\item[(i)] $R$ is Clifford $t$-regular (respectively, $t$-stable);
\item[(ii)]  $R_M$ is Clifford $t_M$-regular (respectively, $t_M$-stable) for each $M\in \tmax(R)$ and $R$ has $t$-finite character.
\end{enumerate}
\end{cor}
\begin{proof} (i) $\ra$ (ii) $R_M$ is Clifford $t_M$-regular (respectively, $t_M$-stable) by Corollary \ref{loc-t}. Since a Clifford $t$-regular domain satisfies condition (iv) of Proposition \ref{wM} \cite[Lemma 4.4]{HK}, then $R$ has $t$-finite character.

(ii) $\ra$ (i) If $R$ has $t$-finite character it is weakly Matlis. Hence we can apply  Corollary \ref{localwM}.
\end{proof}


\end{document}